\documentclass[11pt]{amsart}
\usepackage{amsmath,amssymb,amsthm,color,tikz}
\usepackage{xypic}
\usepackage{graphicx}
\input xy
\xyoption{all}

\newtheorem{theorem}{\sc Theorem}[section]
\newtheorem{lemma}[theorem]{\sc Lemma}
\newtheorem{proposition}[theorem]{\sc Proposition}
\newtheorem{corollary}[theorem]{\sc Corollary}
\newtheorem{definition}[theorem]{\sc Definition}
\newtheorem{conjecture}[theorem]{Conjecture}
\theoremstyle{remark}

\newcommand{\be}{\begin{equation}}
\newcommand{\ee}{\end{equation}}

 \DeclareMathOperator{\Gal}{Gal}

\DeclareMathOperator{\Z}{\mathbb{Z}}
\DeclareMathOperator{\F}{\mathbb{F}}
\DeclareMathOperator{\ram}{Ram}
\DeclareMathOperator{\q}{\mathbb{Q}}
\DeclareMathOperator{\Aut}{Aut}

\definecolor{darkgreen}{rgb}{0.0,0.5,0.0}
\definecolor{darkblue}{rgb}{0.0,0.0,0.3}
\definecolor{nicosred}{rgb}{0.65,0.1,0.1}
\definecolor{light-gray}{gray}{0.7}
\numberwithin{equation}{section}

\allowdisplaybreaks[1]

\begin{document}

\title{Minimal Ramification and the Inverse Galois Problem over the Rational Function Field $\mathbb{F}_p(t)$}
\author[M.~DeWitt]{Meghan DeWitt}
\address{Meghan De Witt\\ Brigham Young University\\ Provo, UT  84602\\ USA.}
\email{megdewitt@gmail.com}
\email{dewitt@math.byu.edu}
\date{\today}
\begin{abstract}
The inverse Galois problem is concerned with finding a Galois extension of a field $K$ with given Galois group.  In this paper we consider the particular case where the base field is $K=\F_p(t)$.  We give a conjectural formula for the minimal number of primes, both finite and infinite, ramified in $G$-extensions of $K$, and give theoretical and computational proofs for many cases of this conjecture.
\end{abstract}
\maketitle

\section{Introduction}
In 1892 Hilbert proposed the first systematic approach to solving the question of which finite groups occur as Galois groups over the rational numbers $\q$, using his Irreducibility Theorem to consider the problem over $\q(t)$ \cite[Page v]{MM}.  Since his work, the question of which finite groups occur as Galois groups over $\q$, then later over any field $K$, has been studied extensively.

What types of restrictions can we place on the field extension for a given group $G$?  Can we produce a $G$-extension that is ramified at a specific prime, or unramified outside of a set of primes?  We consider the case, for a fixed finite group $G$ and global field $K$, of finding the minimal number of primes that will ramify in any $G$-extension of $K$.

Let $K$ be a global field, and define $$\ram _K(G):= \mathop{\mbox{min}}_{\text{Gal}(L/K)\cong G} {\# \{\text{places that ramify in } L/K\}}.$$  Let $d(G)$ denote the minimum number of generators for the group $G$.  For completeness, we set $d(\{1\})=0$.  Also, we define $p(G)$ to be the normal subgroup generated by the elements of $p$-power order.

Work of Boston and Markin \cite{BM} explored this question in the case where $K=\q$, giving the expected minimal number of ramified primes as $d\left(G^{ab}\right)$.  Harbater \cite{Harb} and Raynaud \cite{Ray} proved Abhyankar's Conjecture, which covers the same situation for $K=\overline{\F_p}(t)$, giving the minimal number of ramified primes as $d\left(G/p(G)\right)+1$.  Grothendieck explored the same issue with his work on the fundamental group of the punctured projective line \cite{Groth1,Groth2}.

In the following, we consider the case where $K=\F_p(t)$.  Further, we restrict attention to geometric extensions, meaning we do not allow any extension of the field of constants, which gives us a trivial lower bound $\ram_K(G)\geq 1$ for any non-trivial group $G$ \cite[Page 106]{Stich}.  In addition, we know that class field theory will provide us with a better lower bound; see Theorem \ref{Abel-p}, Theorem \ref{Abel}, and Corollary \ref{minimality} contained herein.

Based on this related work and numerous families of examples, we have the following conjecture:

\begin{conjecture}[Restricted inverse Galois problem over Function Fields]  \label{Restricted IGP}
If $G$ is a nontrivial finite group then there exists a $G$-extension of $\mathbb{F}_p(t)$ and $$\ram_{\F_p(t)}(G)=\begin{cases} d+1 &  \text{ if $p\mid |G^{ab}|$} \\ \max(d,1) &  \text{ if $p \nmid |G^{ab}|$}  \end{cases}$$  where $d=d((G/p(G))^{ab})$.
\end{conjecture}

We provide proofs for abelian group, groups of prime power order, and nilpotent groups, as well as several other families of examples.

\section{Basic Results}

We begin with the basic building blocks, namely abelian groups and $p$-groups.  When tackling the abelian groups, it is useful to have a function field analogue of the Kronecker-Weber Theorem.  To do this, we turn to Carlitz-Hayes Theory \cite{Hayes}.

For each polynomial $M\in \F_p[t]$ we define the Carlitz polynomial $[M](x)$ with coefficients in $\F_p[t]$ recursively:
\begin{align*} [1](x)& = x\\ [t](x)&=x^p+t x \\  \left[t^n\right](x)&=[t]\left(\left[t^{n-1}\right](x)\right) \\ \left[c_nt^n+\cdots +c_1 t+c_0\right](x)&=c_n\left[t^n\right](x)+\cdots+c_1[t](x)+c_0[1]x. \end{align*} In addition, we use a similar definition with $1/t$ in place of $t$: $$\left[\frac{1}{t}\right](x)=x^p+\frac{x}{t}.$$

\vspace{12pt}
Let $K$ be a field extension of $\F_p(t)$.  We make $K$ into an $\F_p[t]$-module by letting $\F_p[t]$ act on $K$ through the Carlitz polynomials: $$M\cdot \alpha = [M](\alpha).$$  Define $$\Lambda_M=\{\lambda \in \overline{\F_p(t)} \mid [M](\lambda)=0\}.$$ Then $\F_p(t,\Lambda_M)/\F_p(t)$ is an abelian extension called a cyclotomic function field extension.  Note that $\Lambda_M$ is a free $\F_p[t]/M$-module of rank 1.  Choose $\sigma\in \Gal (\F_p(t,\Lambda_M)/\F_p(t))$ and let $\lambda$ be a generator of $\Lambda_M$.  Then $\sigma$ acts as $A$ on $\lambda$ for some $A\in \left(\F_p[t]/M\right)^\times$, and $\sigma$ acts by the Carlitz action $[A]$ on all the elements of $\Lambda_M$.  We write $A$ as $A_\sigma$.  Then define $$\Phi(M)=|\left(\F_p[t]/M\right)^\times|.$$

\begin{theorem}[Carlitz] The map $\sigma \longmapsto A_\sigma$ is then an isomorphism $$\Gal (\F_p(t,\Lambda_M)/\F_p(t)) \longrightarrow (\F_p[t]/M)^\times.$$
\end{theorem}

\begin{theorem}[Hayes] Every finite abelian extension of $\F_p(t)$ lies in $$\F_{p^s}\left(t,\Lambda_M,\Lambda_{1/t^n}\right)$$ for some $s\geq 1$, $n\geq 1$, and $M\in \F_p[t]$, where $\Lambda_{1/t^n}$ is the set of roots of the Carlitz polynomial $\left[1/t^n\right](x)$ built with $1/t$ in place of $t$.
\end{theorem}

\subsection{$p$-groups}

We first consider the case where $G$ is an abelian $p$-group.

\begin{theorem}\label{Abel-p} If $G$ is a nontrivial finite abelian $p$-group, then Conjecture \ref{Restricted IGP} holds.  Namely, there exists a $G$-extension of $\mathbb{F}_p(t)$ ramified at exactly 1 prime (counting the infinite prime), and there are no unramified $G$-extensions.
\end{theorem}

\begin{proof}

As we are only considering the geometric case, minimality is immediate.  Note that
\begin{align*} H_{n,p} &= \Gal\left(\F_p\left(t,\Lambda_{1/t^{n+1}}\right)/\F_p(t)\right) \\
                       &\cong \left(\F_p\left[1/t\right]/\left(1/t^{n+1}\right)\right)^\times \\
                       &\cong \{ 1+a_1+\cdots +a_nt^n \mid a_i\in \F_p \} \times \F_p^\times \end{align*}
by Hayes \cite{Hayes}.  Then by Lemma 4.4 of Koch \cite{Koch}, we have that $$H_p=\varprojlim{H_{n,p}/\F_p^\times}$$ where $H_p$ is the Galois group of the maximal $p$-extension of $\F_p(t)$, is a free abelian pro-$p$ group on countably many generators.  Note that by construction only the infinite prime ramifies as ramification of finite primes is contained in an extension of the form $\mathbb{F}_p(t,\Lambda_M)$ for abelian groups.  Then every finite abelian $p$-group appears as a quotient of this group.

\end{proof}

To understand  general $p$-extensions, let $\overline{k}_p$ be the maximal $p$-extension of a field $k$ ramified only at infinity.  Denote $G_{k,p}=\Gal\left(\overline{k}_p/k\right)$.

\begin{theorem}\cite[p. 93, Thm 9.1]{Koch}\label{K} If $k$ is a field of characteristic $p$, then $G_{k,p}$ is a free pro-$p$ group with generator rank $$\dim_{\F_p}{k^+/\mathfrak{p}(k^+)}$$ where we have put $$\mathfrak{p}(x)=x^p-x.$$

\end{theorem}

\begin{theorem}\label{p} If $G$ is a nontrivial finite $p$-group, then Conjecture \ref{Restricted IGP} holds.  Namely, there exists a $G$-extension of $\mathbb{F}_p(t)$ ramified at exactly 1 prime (counting the infinite prime), and there are no unramified $G$-extensions.
\end{theorem}

\begin{proof}

Consider the extension $\overline{k}_p/k$ defined above where $k=\F_p(t)$.  Since every $p$-extension has a nontrivial abelian subextension, by Theorem \ref{Abel-p} and Carlitz-Hayes Theory we know this subextension must be contained in a field of the form $$\F_{p^s}\left(t,\Lambda_{1/t^n}\right)/\F_p(t)$$ which is only ramified at infinity.

Conversely, $\overline{k}_p/k$ cannot be ramified at a prime other than the prime at infinity, by its Artin-Schreier construction.  Thus $G_{k,p}$ is only ramified at the prime at infinity.  Since it is free pro-$p$ on countably many generators by Theorem \ref{K}, every finite $p$-group occurs as a subextension.

\end{proof}

\subsection{Abelian groups}

We now turn our attention to abelian groups:

\begin{theorem}\label{Abel} If $G$ is a nontrivial finite abelian group, then Conjecture \ref{Restricted IGP} holds.
\end{theorem}

\begin{proof}

Write $d=d(G/p(G))$.

First, we consider the case where $|G|$ is prime to $p$.  Then write $$G\cong \mathbb{Z}/n_1\mathbb{Z} \times \cdots \times \mathbb{Z}/n_d\mathbb{Z}.$$  Then for each $i$ we can choose a nonzero irreducible $M_i\in \mathbb{F}_p[t]$ such that $$\Phi(M_i)\equiv 0 \pmod{(p-1)n_i}$$  and the $M_i$ are distinct and nonassociate. It follows that $\mathbb{Z}/n_i\mathbb{Z}$ is isomorphic to a quotient of $(\mathbb{F}_p[t]/M_i)^{\times}$.  Thus, by taking a direct product, $G$ is isomorphic to $\Gal(K/\mathbb{F}_p(t))$ where $K$ is a subfield of the compositum of the $M_i$th cyclotomic function fields $$\mathbb{F}_p(t,\Lambda_{M_i})/\mathbb{F}_p(t)$$ for $1\leq i \leq d$.

Note that since each  $M_i$ is irreducible and pairwise nonassociate, $K$ is ramified at exactly $d$ finite primes.  Also, since by a Theorem 3.2 in \cite{Hayes} $e_\infty=p-1$ in the full $M_i$th cyclotomic field, by our choice of $M_i$, $K$ is a subfield of the compositum of the real portion of the $M_i$th cyclotomic fields (meaning the portion not ramified at the infinite prime).  Thus, there are $d$ primes ramified.

Conversely, suppose $K/\mathbb{F}_p(t)$ is a geometric extension with Galois group $G$ and is ramified at the finite primes $\pi_1,\ldots, \pi_k$ and possibly at the infinite prime, and no others.  By Carlitz-Hayes, $K$ is a subfield of a cyclotomic function field $L=\mathbb{F}_{p}(t,\Lambda_M,\Lambda_{1/t^n})$ for some $n\geq 1$ and $$M=\pi_1^{r_1}\cdots \pi_k^{r_k}.$$

If $K$ is tamely ramified at infinity, then $K$ is a subfield of $L^{+}=\mathbb{F}_p(t,\Lambda_M)$.  Then $G$ is isomorphic to a quotient of $$(\mathbb{F}_p[t]/\pi_1^{r_1})^{\times}\times \cdots \times (\mathbb{F}_p[t]/\pi_k^{r_k})^{\times}$$ and hence has less than or equal to $k$ generators.

If $K$ is not tamely ramified at infinity, then let $K^{+}=K\cap L^{+}$, and $G^{+}=\Gal(K^{+}/\mathbb{F}_p(t))$.  Then, as above, $G^{+}$ has at most $k$ generators.  Thus $G$ has less than or equal to $k$ generators.

Now, suppose $|G|$ is not prime to $p$ (See Figure 1).  We obtain the desired extension by using the above for the prime-to-$p$ part and then using Theorem \ref{p} to obtain the $p$-portion.  Since $G$ is abelian, we can then realize $G$ by taking the compositum of these two fields.

\begin{figure}\label{Figure1}\caption{Breakdown of an Abelian Extension} $$\xymatrix@C=10pt@R=12pt{
  & & &  &  &   \F_p\left(t,\Lambda_M,\Lambda_{1/t^n}\right)\ar@{-}[dr]\ar@{-}[dl] &  \\
  \F_p\left(t,\Lambda_{1/t^n}\right)\ar@{-}[drr] & & &  &  L\ar@{-}[dr]^{p(G)}\ar@{-}[dll]_{G/p(G)} &  & \F_p(t,\Lambda_M)\ar@{-}[dl]   \\
  & & N\ar@{-}[drr]_{p(G)} &  &  &  K\ar@{-}[dl]^{G/p(G)}  &  \\
  & & &  &  \F_p(t) &   & }
   $$
 \end{figure}
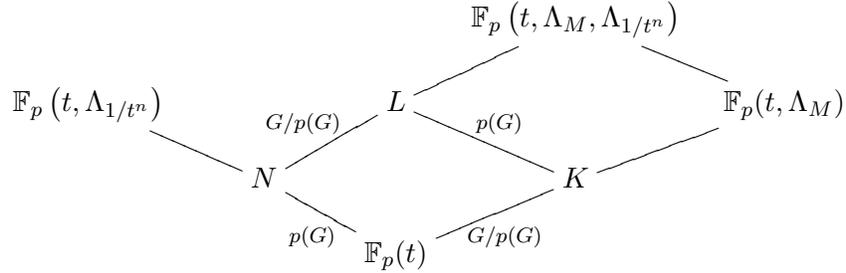

Conversely, suppose we have a geometric $G$-extension $K/\mathbb{F}_p(t)$.  By the above, the prime-to-$p$ part has at least $d$ ramified primes.  Thus it remains to show that every abelian $p$-extension of $\F_p(t)$ is ramified at the infinite prime.  However, this follows from Carlitz-Hayes theory and the fact that the degree of cyclotomic extensions of the form $\F_p(t,\Lambda_M)/\F_p(t)$ are always prime-to-$p$.

\end{proof}

Then we have also established the following:

\begin{corollary}\label{minimality}
$$\ram_{\F_p(t)}(G)\geq\begin{cases} d+1 &  \text{ if $p\mid |G^{ab}|$} \\ \max(d,1) &  \text{ if $p \nmid |G^{ab}|$}  \end{cases}$$
holds for any nontrivial finite group $G$.
\end{corollary}

\section{Semiabelian Groups}

Theorem \ref{p} established that Conjecture \ref{Restricted IGP} holds for $p$-groups; we now  consider $l$-groups, where $l$ is a prime different from $p$.  Specifically, we consider the case of semiabelian $l$-groups, which leads to a more general result.  The following are from \cite{KS}:

\begin{definition}
Let $G, H$ be finite groups.  We define the \textit{wreath product} $H\wr G$ of $H$ with $G$ to be the semidirect product $H^{|G|}\rtimes G$, where $H^{|G|}$ is the direct product of $|G|$ copies of $H$, with $G$ acting on $H^{|G|}$ by permuting the copies of $H$ as in the regular (Cayley) representation of $G$.

We define the \textit{wreath length} of a group $G$ to be the smallest positive integer $r$ such that there are finite cyclic groups $C_1,\ldots, C_r$ and an epimorphism $$C_1\wr (C_2 \wr ( \cdots \wr C_r)\cdots )\twoheadrightarrow G,$$ if such a number exists, and denote it $wl(G)$.
\end{definition}

\begin{definition}
A finite group $G$ is called \textit{semiabelian} if there exists a sequence $$G_0=\{1\}, G_1,\ldots,G_n=G$$ such that $G_i$ is a homomorphic image of a semidirect product $A_i \rtimes G_{i-1}$ with $A_i$ abelian, $i=1,\ldots,n$.
\end{definition}

\begin{proposition}\cite{KS} \label{wr}
For any prime $l$, the smallest family containing all cyclic $l$-groups that is closed under homomorphic images, direct products, and wreath products is the family of semiabelian $l$-groups.  Furthermore, for the elements of this family the wreath length is defined and is exactly $d(G)$.
\end{proposition}

\begin{lemma}\cite{KS2}\label{2} Let $K$ be a global field.  Let $\mathfrak{a}_1,\ldots, \mathfrak{a}_s\in \mathcal{I}_\mathfrak{p}$ such that their images in $\mathcal{C}l_K^{(l)}$ form a minimal set of generators.  Let $l_{m_i}$ be the order of $\overline{\mathfrak{a}_i}$ and $\mathfrak{a}_i^{l^{m_i}}=(a_i)\in\mathcal{P}_\mathfrak{p}$. Let $K^{\prime\prime}=K\left(\zeta_{l^m},\sqrt[l^m]{\xi},\sqrt[l^{m_i}]{a_i}\mid \xi\in U_K,i=1,\ldots, r\right)$, where $U_K$ is the group of units of $K$ and $\mathfrak{p}$ a prime of $K$ which splits completely in $K^{\prime\prime}$.  Then there is a cyclic $l^m$-extension of $K$ that is totally ramified at $\mathfrak{p}$ and is not ramified at any other prime of $K$.
\end{lemma}

\begin{corollary}\label{3} Let $K/\F_p(t)$ be a function field, $n$ a positive integer with $(n,p)=1$.  Then there exists a finite extension $K^{\prime\prime\prime}$ of $K$ such that if $\mathfrak{p}$ is any prime of $K$ that splits completely in $K^{\prime\prime\prime}$, then there exists a cyclic extension $L/K$ of degree $n$ in which $\mathfrak{p}$ is totally ramified and $\mathfrak{p}$ is the only prime of $K$ that ramifies in $L$.
\end{corollary}

\begin{proof} We modify the number field case found in \cite{KSN}.

Let $n=\prod_{q}{q^{m(q)}}$ be the decomposition of $n$ into primes.  Let $K^{\prime\prime}=K^{\prime\prime}(q)$ by taking $m=m(q)$ in Lemma \ref{2}, then let $K^{\prime\prime\prime}$ be the compositum of the fields $K^{\prime\prime}(q)$.  Let $L(q)$ be the cyclic extension of degree $q^{m(q)}$ provided by Lemma \ref{2}.  The compositum $L=\prod{L(q)}$ has the desired property.
\end{proof}

\begin{theorem}\label{semi} Let $G$ be a finite semiabelian group, of order prime to $p$.  Then there exists a tamely ramified extension $K/\F_p(t)$ with Galois group $G$ in which at most $d(G)$ primes ramify.
\end{theorem}
Note, this does not go as far as Conjecture \ref{Restricted IGP}, as we are finding $d(G)$ not $d\left(G^{ab}\right)$ ramified primes (here $p(G)=\{1\}$).  However, when the two are equal, as is the case for all $l$-groups by Burnside's Basis Theorem \cite[p. 46]{Koch}, we will then have shown Conjecture \ref{Restricted IGP}.
\begin{proof} Again, this is a modification of the number field case found in \cite{KSN}.
By definition, $G$ is a homomorphic  image of $C_1 \wr(C_2\wr \cdots \wr C_r)$, $r=wl(G)=d(G)$.  Thus we may assume $G\cong C_1\wr (C_2 \wr \cdots \wr C_r)$.  We then induct on $r$.

For $r=1$, $G$ is cyclic.  Then we are done by Theorem \ref{Abel}.

Now, assume the theorem holds for $r-1$.  let $K_1/\F_p(t)$ be a tamely ramified Galois extension with $\Gal(K_1/\F_p(t))\cong C_2\wr (C_3\wr \cdots \wr C_r)$ such that the ramified primes in $K_1$ are a subset of $\{\mathfrak{q}_2,\ldots, \mathfrak{q}_r\}$.  By Corollary \ref{3}, there exists a field $K_1^{\prime\prime\prime}$, the field supplied for $K_1$ by Corollary \ref{3}, and a prime $\mathfrak{q}=\mathfrak{q}_1$ which splits completely in $K_1^{\prime\prime\prime}$.  Let $\mathfrak{p}=\mathfrak{p}_1$ be a prime of $K_1$ dividing $\mathfrak{q}$.  Then there exists a cyclic extension $L/K_1$ with $\Gal(L/K_1)\cong C_1$ in which $\mathfrak{p}$ is totally ramified and in which $\mathfrak{p}$ is the only prime of $K_1$ which ramifies in $L$.

Let $\{\sigma(\mathfrak{p})\mid \sigma\in\Gal(K_1/\F_p(t))\}$ be the $|G_1|$ distinct conjugates of $\mathfrak{p}$ over $K_1$.  Let $M$ be the compositum of $\sigma(L)$, as $\sigma$ runs over $\Gal(K_1/\F_p(t))$.  Then it follows that the fields $\{\sigma(L) \mid \sigma \in \Gal(K_1/\F_p(t))\}$ are linearly disjoint over $K_1$, and so $$\Gal(M/\F_p(t))\cong C_1\wr G_1\cong G.$$  With the ramified primes of $M/\F_p(t)$ a subset of $\{\mathfrak{q}_1,\ldots,\mathfrak{q}_r\}$.  Then $M$ satisfies the conditions of the theorem.
\end{proof}

\section{Small Group Examples}

Before we proceed further, it is useful to examine several concrete examples.  In the process of proving these cases, it is advantageous to have a definitive method for determining the ramification at the infinite prime.  We use the following criterion when dealing with explicit examples:

\begin{lemma}[Ramification at Infinity] \label{inf}
Suppose $f(x)$ has a splitting field $K/\F_p(t)$.  The following procedure is sufficient to determine if the infinite prime $(1/t)$ ramifies in the extension $K/\F_p(t)$. \begin{enumerate}
\item Substitute $xt^{b}$ for $x$.
\item Divide by $t^{b\cdot \deg{f}}$.
\item Mod out by $1/t$ to get $g(x)$.
\item Determine if $g(x)$ has a repeated root.
\end{enumerate}  Here, $b$ is the smallest integer such that after step 2 there are no positive powers of $t$ left.
\end{lemma}

\begin{proof}
We employ the first two steps to obtain a polynomial in $1/t$ instead of $t$.  As in \cite[Lemma 3.5.3, Cor 3.5.11]{Stich}, we can then check the ramification by reducing mod $1/t$.
\end{proof}

One of the most useful tools we have in the function field case is an analogue of Schinzel's Hypothesis-H that allows us to produce irreducible polynomials satisfying certain properties:

\begin{theorem}[Pollack, \cite{Pollack}]\label{Pollack}
Let $n$ be a positive integer.  Let $f_1(x), \ldots , f_r(x)$ be nonassociate irreducible polynomials over $\F_q$ with the degree of the product $f_1\cdots f_r$ bounded by $B$.  The number of univariate monic polynomials $g$ of degree $n$ for which all of $f_1(g(t)),\ldots, f_r(g(t))$ are irreducible over $\F_q$ is $$\frac{q^n}{n^r}+ O_{n,B}\left(q^{n-\frac{1}{2}}\right)$$ provided $gcd(q,2n)=1$.
\end{theorem}

Pollack gives explicit upper and lower bounds here for most $q$.  In particular, if we let $C$ be the number of such $g$, we have $$C\geq \left(q^{n-1}-4n^2q^{n-2}\left(1 +{B\choose 2}\right)\right)\left(\frac{q}{n^r}-\frac{2}{n^r}\left(q^{1/2}+1+n!^B\right)-(n-1)B\right) $$ when $q$ is sufficiently large.  Specifically, $q$ must be large enough to satisfy $$q>4n^2\left(1+{B\choose 2}\right).$$  Taken together, these then ensure that $C>0$.  Thus, when $q$ is sufficiently large we can always find at least one such $g$.

With this tool in hand, we now proceed to outline several specific examples that provide support for Conjecture \ref{Restricted IGP}.

\subsection{$D_8$}

We begin with a group covered in the previous theorems (specifically Theorem \ref{p} and Theorem \ref{semi}), namely, the dihedral group of order 8.  It is useful, however, to consider this group explicitly.

\begin{theorem} Conjecture \ref{Restricted IGP} holds for $G=D_8$, the dihedral group of order 8.  Namely, there exists a $D_8$-extension of $\mathbb{F}_p(t)$ ramified at exactly 2 primes (counting the infinite prime) when $p\neq 2$, or one prime when $p=2$.  Moreover, there is no such extension ramified at fewer primes.
\end{theorem}

\begin{proof}
When $p=2$, this falls under Theorem \ref{p}.  We now consider the case $p\neq 2$.

Every dihedral extension can be defined by a polynomial of the form $f(x)=x^4+ax^2+b$.  Further, any such choice of $a,b$ will yield a dihedral extension when $b$, $a^2-4b$, and $b(a^2-4b)$ are all not squares.  Then, the discriminant of $f$ is $$16b(a^2-4b)^2,$$ so it is sufficient to check the ramification in each each of the fields defined by $x^2-b, x^2-(a^2-4b), \text{ and } x^2-b(a^2-4b)$.

Now, based on the abelian theory above, it is sufficient to find $a,b$, with $b$ irreducible, satisfying the non-square conditions described, and then show that the infinite prime does not ramify.

We use Lemma \ref{inf} to test the ramification at the infinite prime.  In most cases, $(1/t)$ will still ramify, but if $2\deg a=\deg b$, and $A^2-4B$ is not a square, $f$ is unramified at infinity.  Here $A$ and $B$ are the leading coefficients of $a$ and $b$, respectively.

Choose $a$ such that $A^2-4$ is not a square.  Choose $d\in \F_p[t]$, a square element.  Let $$g_1(x)=a^2-4x,\hspace{6pt} g_2(x)=x(a^2-4x)-d.$$ Then use Theorem \ref{Pollack} to choose a monic irreducible $b\in \F_p[t]$ such that $g_1(b)$ and $g_2(b)$ are still irreducible.  Thus, by choice of $b$, $f(x)$ must define a dihedral extension ramified at only two finite primes, and by choice of $a$, $f(x)$ is unramified at infinity.  Note, Theorem \ref{Pollack} will give the existence of such a $b$ as long as $$p>4n^2\left(1+{B\choose 2}\right).$$  As we may let $n=4$, Theorem \ref{Pollack} then produces the desired $b$ for $p>256$.  We provide explicit examples for the remaining primes in Table 1.  (Note that for $p=3$, we do not give an irreducible $b$, but instead $b=2(t+2)^4$.  It is not a perfect square, so the extension is still dihedral, and it is the power of a single prime, so only one prime ramifies in the extension generated by $x^2-b$.)

For minimality, note that any dihedral extension has a Klein-4 subextension, and hence by the abelian theory above must have at least two primes that ramify.
\end{proof}

\begin{center}
\begin{table}[h!b!p!]\label{TableS3}
\caption{\tiny{We present the explicit examples for small primes that have Galois group $D_8$.  Here, we have a defining polynomial $$f(x)=x^4+ax^2+b$$  We give the values of $A$, $A^2-4B$, $a$, and $b$, where $A$ is the leading coefficient of $a$, and $B$ is the leading coefficient of $b$.}}
\tiny{\begin{tabular}{|c|c|c|c|c||c|c|c|c|c|}
\hline
$p$ & $A$ & $A^2-4B$ & $a$ & $b$ & $p$ & $A$ & $A^2-4B$ & $a$ & $b$\\
\hline
\hline
  3 & 2 & 2 & $2t^2+t+1 $ & $2t^4+t^3+t+2 $ & 109 & 11 & 8 & $11t^2+3$ & $t^4+5t+1$  \\
  \hline
  5 & 1 & 2 & $t^2+2$ & $t^4+2$  & 113 & 32 & 3 & $32t^2+3$ & $t^4+t+1$ \\
  \hline
  7 & 3 & 5 & $3t^2+3t$ & $t^4+x+1$ & 127 & 3 & 5 & $3t^2+3$ & $t^4+4t+1$  \\
  \hline
  11 & 1 & 8 & $t^2+3t+1$ & $t^4+4t+1$ & 131 & 55 & 8 & $55t^2+6$ & $t^4+t+1$ \\
  \hline
  13 & 3 & 5 & $3t^2+7$ & $t^4+t+1$ & 137 & 12 & 3 & $12t^2$ & $t^4+t+1$  \\
  \hline
  17 & 3 & 5 & $3t^2$ & $t^4+3t+1$ & 139 & 4 & 12 & $4t^2+1$ & $t^4+3t+1$  \\
  \hline
  19 & 5 & 2 & $5t^2+1$ & $t^4+6t+1$ & 149 & 4 & 12 & $4t^2+1$ & $t^4+16t+1$  \\
  \hline
  23 & 3 & 5 & $3t^2+3$ & $t^4+4t+1$ & 151 & 37 & 6 & $37t^2+3$ & $t^4+5t+1$  \\
  \hline
  29 & 8 & 2 & $8t^2+7$ & $t^4+3t+1$ & 157 & 3 & 5 & $3t^2+1$ & $t^4+t+1$  \\
  \hline
  31 & 4 & 12 & $4t^2$ & $t^4+t+1$ &  163 & 13 & 2 & $13t^2+3$ & $t^4+t+1$ \\
  \hline
  37 & 3 & 5 & $3t^2+3$ & $t^4+5t+1$ & 167 & 3 & 5 & $3t^2$ & $t^4+3t+1$  \\
  \hline
  41 & 4 & 12 & $4t^2+1$ & $t^4+t+1$ & 173 & 51 & 2 & $51t^2$ & $t^4+7t+1$   \\
  \hline
  43 & 7 & 2 & $7t^2+5$ & $t^4+5t+1$ & 179 & 38 & 8 & $38t^2+1$ & $t^4+t+1$  \\
  \hline
  47 & 3 & 5 & $3t^2+3$ & $t^4+5t+1$ & 181 & 83 & 7 & $83t^2+3$ & $t^4+8t+1$ \\
  \hline
  53 & 1 & 50 & $t^2$ & $t^4+3t+1$ & 191 & 46 & 11 & $46t^2+6$ & $t^4+4t+1$   \\
  \hline
  59 & 1 & 56 & $t^2$ & $t^4+t+1$ & 193 & 3 & 5 & $3t^2+8$ & $t^4+8t+1$   \\
  \hline
  61 & 5 & 21 & $5t^2$ & $t^4+3t+1$ & 197 & 91 & 3 & $91t^2+12$ & $t^4+11t+1$  \\
  \hline
  67 & 26 & 2 & $26t^2$ & $t^4+t+1$ & 199 & 87 & 3 & $87t^2+1$ & $t^4+t+1$ \\
  \hline
  71 & 19 & 2 & $19t^2+4$ & $t^4+8t+1$ &  211 & 46 & 2 & $46t^2+1$ & $t^4+t+1$  \\
  \hline
  73 & 3 & 5 & $3t^2$ & $t^4+4t+1$ &  223 & 26 & 3 & $26t^2+4$ & $t^4+13t+1$ \\
  \hline
  79 & 13 & 7 & $13t^2+3$ & $t^4+6t+1$ &  227 & 3 & 5 & $3t^2+6$ & $t^4+4t+1$ \\
  \hline
  83 & 3 & 5 & $3t^2+6$ & $t^4+10t+1$ &  229 & 34 & 7 & $34t^2+4$ & $t^4+21t+1$  \\
  \hline
  89 & 10 & 7 & $10t^2$ & $t^4+3t+1$ &   233 & 3 & 5 & $3t^2+7$ & $t^4+8t+1$  \\
  \hline
  97 & 3 & 5 & $3t^2+9$ & $t^4+4t+1$ &  239 & 49 & 7 & $49t^2+4$ & $t^4+5t+1$ \\
  \hline
  101 & 39 & 2 & $39t^2+5$ & $t^4+t+1$ &  241 & 16 & 11 & $16t^2+6$ & $t^4+11t+1$ \\
  \hline
  103 & 25 & 3 & $25t^2+1$ & $t^4+5t+1$ & 251 & 99 & 8 & $99t^2+8$ & $t^4+15t+1$ \\
  \hline
  107 & 3 & 5 & $3t^2+1$ & $t^4+14t+1$ &    &  &  &  &   \\
\hline
\end{tabular}}
\end{table}
\end{center}

\subsection{$S_3$}

We now examine the symmetric group on three elements.  This group, being the smallest nonabelian example, provides much direction for further work. Consider the polynomial $$f(x)=x^3-uwx-u^2, \hspace{10pt} w,u\in \F_p[t]$$  where $u$ and $w$ are relatively prime.  Note, the discriminant of $f(x)$ is $d=4u^3w^3-27u^4$.  For a field $K$, let $h(K)$ denote the divisor class number, namely the order of the finite portion of the class group. Then

\begin{theorem}\cite{LZ}\label{LZ} When $p>3$, if $d$ is not a square in $\F_p(t)$, then $3 \mid h(\F_p(t)(\sqrt{d}))$.  Conversely, every quadratic function field whose divisor class number is divisible by 3 is given in this way by some $u$ and $w$.
\end{theorem}

\begin{lemma}\label{S3}Let $K/\F_p(t)$ be a quadratic extension.  Then $3|h(K)$ if and only if there exists an unramified geometric cyclic extension $L/K$ such that $L$ is Galois over $\F_p(t)$ with Galois group isomorphic to $S_3$.
\end{lemma}
\begin{proof}
By class field theory we have an unramified degree 3 extension of $K$ for every quotient of order 3 of the class group.  At least one of these must be Galois over $\F_p(t)$.  Thus it will have Galois group $S_3$ or $\Z/6\Z$.  However, the latter would then yield an unramified cubic extension of $\F_p(t)$ which is impossible.
\end{proof}

\begin{lemma}\cite{LZ}\label{lz2} Consider $f(x)$ defined above.  If $f$ is irreducible, let $K=\F_p(t)(\sqrt{d})$ and $L/\F_p(t)$ be the splitting field of $f$.  $L$ is unramified over $K$ if and only if $v=w^3$ and $\deg u <\deg v$ or $3\mid \deg u$.

\end{lemma}

We will use this result to control the ramification for general $p$.

\begin{theorem}\label{S32}\label{lz3} Conjecture \ref{Restricted IGP} holds for $G=S_3$ and $p\equiv 0,1 \pmod 3$.  Namely, there exists an $S_3$-extension of $\F_p(t)$ ramified at only one prime.  Moreover, there is no such extension ramified at fewer primes.
\end{theorem}

\begin{proof}
Note, minimality is ensured by Corollary \ref{minimality} or by Hermite's Theorem for function fields \cite[Theorem III.2.16]{Neukirch2}.  For existence, we consider several cases.

$(p=2)$ \hspace{5pt}  Note that $G/p(G)\cong \{1\}$, which has zero generators, according to our convention.  Then we are expecting the minimal number of ramified primes to be 1.

As an example of the existence of such an extension, consider the field defined by $$f(x)=x^2+x+(t+1)^3.$$  Computations confirm that only the infinite prime ramifies.  This field has class group $\Z\times \Z/3\Z$, thus there is a cubic extension of the quadratic field which is an $S_3$-extension over $\F_p(t)$ by Lemma \ref{S3}.

$(p=3)$ \hspace{5pt}  Note that $G/p(G)\cong \Z/2\Z$, which has one generator.  Then we are expecting the minimal number of ramified primes to be 1.

Consider the splitting field of the polynomial $$f(x)=x^3-(t^2+1)x+(t-1),$$  whose discriminant is $D=(t^2+1)^3$.  Computations confirm that $(t^2+1)$ ramifies, and that it is the only prime that ramifies in the quadratic subextension.  The discriminant is not a square, thus we again retrieve an $S_3$-extension ramified at only one finite prime \cite[p. 612]{Dummit}.

$(p\geq 7,\hspace{6pt} p\equiv1 \pmod 3)$ \hspace{5pt}  Note that $\left(G/p(G)\right)^{ab}\cong \Z/2\Z$, which has one generator.  Then we are expecting the minimal number of ramified primes to be 1.  Let $k=\F_p(t)$ and define $$f(x)=x^3-uwx-u^2$$  where $u,w\in \F_p[t]$ are relatively prime with $\deg u < 3\deg w$ or $3| \deg u$.  Thus by Lemma \ref{lz2} and Theorem \ref{LZ} we have its splitting field is an $S_3$-extension with all the ramification in the quadratic subextension. Then $f$ has discriminant $$d=4u^3w^3-27u^4=u^3(4w^3-27u).$$  Suppose $f$ is irreducible over $k$.  Then work of Li and Zhang \cite[Lemma 2.2]{LZ} implies that $K/k$ will have only one finite ramified prime if we can choose $u,w$ such that $u\in\F_p$ and $w$, nonconstant, such that $\pi=4w^3-27u$ is irreducible.  With such a choice of $u$ and $w$, it is easy to see that $f$ is irreducible: if $f$ factored, it would have a root in $\F_p$[t] but this would then imply that $f$ would have a root in $\F_p$, forcing $w\in \F_p$, which is false.

By Lemma \ref{inf} we see that the infinite prime will always ramify if $\deg w$ is odd.  However, if $\deg w $ is even and $w$ is monic, we arrive at $g(x)=x(x^2-u)$.  This will not have a repeated root so long as we can choose $u\neq 0$ in $\F_p$, since $p\neq 2$.  Thus we can choose a polynomial where the infinite prime does not ramify, so we have an extension ramified at only one prime.

Now, it suffices to show that there always is such a choice of $u$ and $w$, with $\pi$ irreducible and $w$ of even degree.  For this we apply Theorem \ref{Pollack}.  As seen previously, it is possible to choose parameters to guarantee a nonzero answer when $p$ is sufficiently large; in this case, when $p>64$ and when $\F_p\neq \F_p^2$; this requires that $p\equiv1 \pmod{3}$  We give explicit examples for the remaining primes in Table 2.

\end{proof}

\begin{center}
\begin{table}[h!b!p!]\label{TableS3}
\caption{\tiny{We present the explicit examples for small primes that have Galois group $S_3$.  Here, we have a defining polynomial $$f(x)=x^3-uwx-u^2$$  The discriminant is then $$d=4u^3w^3-27u^4=u^3(4w^3-27u)$$  We give the value of $u$ and $w$ for each prime, as well as the value of $\pi=4w^3-27u$.}}
\begin{tabular}{|c|c|c|c|}
  \hline
  $p$ & $u$ & $w$ & $\pi$ \\
  \hline
  \hline
  7 & 6 & $t^2+3$ & $4t^6+t^4+3t^2+2$ \\
  \hline
  13 & 2 & $t^2$ & $4t^6+11$ \\
  \hline
  19 & 2 & $t^2+1$ & $4t^6+12t^4+12t^2+7$ \\
  \hline
  31 & 3 & $t^2+1$ & $4t^6+12t^4+12t^2+6$ \\
  \hline
  37 & 2 & $t^2$ & $4t^6+20$ \\
  \hline
  43 & 3 & $t^2+1$ & $4t^6+12t^4+12t^2+9$ \\
  \hline
  61 & 2 & $t^2$ & $4t^6+7$ \\
\hline
\end{tabular}
\end{table}
\end{center}

\section{Dihedral Groups}

We now attempt to generalize the previous results.  Given a finite cyclic group $A$ with $|A|=a$, we can form the semidirect product $A\rtimes \Z/2\Z$, the dihedral group of order $2a$, and denote it by $D_{2a}$.

\begin{lemma}\label{CH}
Suppose $p\neq 2$.  Let $A$ be as above, with $(a,2)=1$ and $(a,p)=1$.  Let $K/\F_p(t)$ be a quadratic extension with only one ramified prime, $N\in \F_p[t]$.  There exists a minimally ramified $A$-extension of $K$, call it $L/K$, with one ramified prime, such that $L/\F_p(t)$ is Galois.
\end{lemma}
\begin{proof}
Carlitz-Hayes theory as defined over $K=\F_p(t)$, can be generalized over any field assuming the embedding $$\alpha:\Gal(K(\Lambda_M)/K)\hookrightarrow (\F_p[t]/M)^\times$$ is an isomorphism.  We refer to the proof of the above fact for $K=\F_p(t)$ and the discussion of generalizations found in \cite{Car1, Car2, Hayes}.  For $M\in \mathcal{O}_K$, define $$\Phi(M)=|(\mathcal{O}_K/M)^\times|.$$  Choose $M\in \mathcal{O}_K$ irreducible such that $$\Phi(M)\equiv 0\pmod{a(p-1)},$$ and for $M|M^\prime$, $M^\prime \in \F_p[t]$, $$\Phi(M^\prime)\equiv 0\pmod{2(p-1)}$$ but $$\Phi(M^\prime)\not\equiv 0 \pmod{2a(p-1)}.$$ Such an $M$ and $M^\prime$ exist by definition of $\Phi$ and the fact that $p\nmid 2,a$.  Note, the proof referenced above shows $\Gal(K(\Lambda_M)/K)\cong (\mathcal{O}_K/M)^\times$ for $M\in \mathcal{O}_K$, and taking the congruences with a $p-1$ term in the modulus guarantees that infinity will not ramify \cite[Theorem 3.2]{Hayes}.  Then $A$ is isomorphic to a quotient of $(\mathcal{O}_K/M)^\times$ and, since by our choice $M$ lies over a prime $M^\prime$ with no inertia, is isomorphic to a quotient of $(\F_p[t]/M^\prime)^\times$.  So with a choice of such an $M$, we can build such an $A$-extension $L/K$.  Further, $L/K$ has one ramified prime by our choice of $M$, namely $M$ itself.

Now, to see that $L/\F_p(t)$ is Galois, it is sufficient to note that $(\F_p[t]/N)^\times$ acts on $(\F_p[t]/M^\prime)^\times$.  But as $M^\prime$ is by construction ramified in $L/\F_p(t)$, and $N$ is the only ramified prime in $K/\F_p(t)$, our choice of $M$ forces $M^\prime =N$.
\end{proof}

Note, the condition that $(a,2)=1$ ensures that we fall into one of two cases.  If $L/\F_p(t)$ is cyclic it has Galois group $A \times \Z/2\Z$ and then it could have been achieved by the Carlitz-Hayes theory shown previously and thus we only have one ramified prime, namely $M=M^\prime$; this was excluded by our choice of $M$ and $M^\prime$.  If $L/\F_p(t)$ is not cyclic but Galois, it has Galois group $A\rtimes \Z/2\Z\cong D_{2a}$, and $M^\prime$ is the only ramified prime.  Then we have:

\begin{theorem}\label{Dn}

Suppose $p\neq 2$.  Let $A$ be a cyclic group of odd prime order $a$, $p\neq a$.  Then $D_{2a}$ satisfies Conjecture \ref{Restricted IGP}.

\end{theorem}

\begin{proof}

According to our conjecture, we are expecting only one ramified prime.  Construct a minimally ramified quadratic extension $K/\F_p(t)$ as in Theorem \ref{Abel}.  Then we use Lemma \ref{CH} to produce an $A$-extension of $K$, minimally ramified, giving an extension $L/\F_p(t)$ having one ramified prime.

To ensure that this extension is not a cyclic extension and that only one prime ramifies, it is enough to carefully pick $M\in \mathcal{O}_K$ and $N \in \F_p[t]$ with $M | N$ but $M\neq N$, as in Lemma \ref{CH}.  Then $N$ is the only prime ramifying in $L/\F_p(t)$, but $L/\F_p(t)$ cannot be cyclic.  Hence we must have $\Gal(L/\F_p(t))\cong A\rtimes \Z/2\Z$, with one ramified prime.  To see that this is in fact $D_{2a}$ we note that when $a$ is an odd prime there is only one homomorphism  $$\varphi: \Z/2\Z \hookrightarrow \Aut(A)$$ and hence only one semidirect product, namely $D_{2a}$.

\end{proof}

\section{Embedding Problems}

\subsection{Embedding theory}

Let $F$ be a field and $\overline{F}$ a separable closure of $F$.  Then define $G_F=\Gal(\overline{F}/F)$.  A Galois extension $N/F$ with group $G$ is then defined by taking a surjection $$\phi:G_F\longrightarrow G$$ and setting $N=\overline{F}^{\ker \phi}$.  Suppose $K/F$ is Galois with group $G$ with associated surjection $\phi:G_F\rightarrow G$.  Consider a group $\tilde{G}$ with an exact sequence $$ \xymatrix{ 1\ar[r] & H\ar[r]_{\iota} & \tilde{G}\ar[r]_{\kappa} & G\ar[r] & 1. }$$  The \textit{embedding problem} $\varepsilon(\phi, \kappa)$ is the question of whether there exists a homomorphism $$\tilde{\phi}:G_F\longrightarrow \tilde{G}$$ which extends $\phi$ via $\kappa$ such that the following diagram commutes:

$$\xymatrix{   &  &  & G_F\ar[d]^{{\phi}}\ar[dl]_{\tilde{\phi}} &  \\
1\ar[r] & H\ar[r]_{\iota} & \tilde{G}\ar[r]_{\kappa} & G\ar[r] & 1.
}
   $$

If $\tilde{\phi}$ is surjective, we say it is a \textit{proper} solution.  Then the corresponding field $\tilde{N}:= \overline{F}^{\ker{\tilde{\phi}}}$ is a $\tilde{G}$-extension of $F$.  We call $H$ the \textit{kernel} of the embedding problem $\varepsilon(\phi, \kappa)$.  We call $\varepsilon(\phi, \kappa)$ \textit{finite} if $\tilde{G}$ is a finite group, \textit{split} if the group extension $\tilde{G}=H\cdot G$ splits, \textit{central} if $H$ lies in the center of $\tilde{G}$, and \textit{Frattini} if $H$ lies in the Frattini subgroup of $\tilde{G}$.  In the case where $F$ is a function field with field, we call $\varepsilon(\phi, \kappa)$ a \textit{geometric embedding problem} if $N:=\overline{F}^{\ker{\phi}}$ is geometric over $F$, or equivalently regular over the field of constants of $F$.


It is possible to break embedding problems up into smaller pieces as follows:

\begin{theorem}[Dentzer, \cite{Dentzer}]
Let the kernel $H$ of an embedding problem $\varepsilon(\phi, \kappa)$ have a decomposition $H=H_1\times H_2$ as a direct product of a normal subgroups of $\tilde{G}$.  For $\tilde{G}_1=\tilde{G}/H_1$ and $\tilde{G}_2=\tilde{G}/H_2$, and the induced epimorphism $\kappa_i:\tilde{G}_i\rightarrow \Gal(N/F)$ the following hold:
\begin{itemize}
\item The embedding problem $\varepsilon(\phi, \kappa)$ is solvable if and only if the embedding problems $\varepsilon(\phi, \kappa_i)$ are solvable.  For the corresponding solution fields, $\tilde{N}=\tilde{N_1}\tilde{N_2}$ holds.
\item $\varepsilon(\phi, \kappa)$ possesses a proper solution if and only if $\varepsilon(\phi, \kappa_i)$ have proper solutions $\tilde{N_1}, \tilde{N_2}$ that are linearly disjoint over $N$.
\item Let $F$ be a function field over $k$.  If $\varepsilon(\phi, \kappa)$ possesses a geometric solution, then so do $\varepsilon(\phi, \kappa_i)$.  If these have geometric solutions $\tilde{N_1}$ and $\tilde{N_2}$ such that the fields $\overline{k} \tilde{N_i}$ are linearly disjoint over $\overline{k} F$, then $\varepsilon(\phi, \kappa)$ also has a geometric solution.
\end{itemize}
\end{theorem}
Further, we have


\begin{theorem}[Dentzer, \cite{Dentzer}]
Let $F$ be a Hilbertian field and $H$ a finite group, being the Galois group of a geometric extension of the rational function field $F(x)$.  Further let $K$ be a finite Galois extension of $F$ with group $G$ and let $\tilde{G}$ be one of the following extensions of $G$:
\begin{itemize}
\item $\tilde{G}=G\times H$
\item $\tilde{G}=G\wr H$
\item Let $H$ be abelian and $\tilde{G}=G\ltimes H$
\end{itemize}
Let $K$ be the corresponding projection in each case.  Then the embedding problem $\varepsilon(\phi, \kappa)$ possesses a proper solution.
If $F=k(t)$ is a rational function field and $K$ is a geometric extension of $F$, then there exists even a proper geometric solution.
\end{theorem}

Define $$\ram (L/K)=\{\mathfrak{p}\in \mathbb{P}(K) \mid \mathfrak{p} \text{ ramifies in } L/K\}.$$  A finite Galois extension $N/F$ is called an \textit{$n$-Scholz extension} if all $\mathfrak{p}\in \ram(N/F)$ and $\tilde{\mathfrak{p}}$ lying over $\mathfrak{p}$ satisfy $\zeta_n\in F_\mathfrak{p}$ and $D(\tilde{\mathfrak{p}}/\mathfrak{p})=I(\tilde{\mathfrak{p}}/\mathfrak{p})$.  We call $\varepsilon(\phi,\kappa)$ an \textit{$n$-Scholz embedding problem} if the fixed field $N$ of $\ker(\phi)$ is an $n$-Scholz extension of $F$ and a solution $\tilde{\phi}$ an \textit{$n$-Scholz solution} if the solution field $\tilde{N}$ is an $n$-Scholz extension of K \cite[Chap 4]{MM}.  We define the \textit{socle} of a Galois $l$-extension $N/F$ with group $G$ to be the maximal elementary abelian intermediate field.  In other words, the fixed field of $\Phi(G)=G^lG^\prime$.

\subsection{General $l$-groups}


Recall that all finite $l$-groups have an upper central series of the form $$1=G_0\leq G_1\leq \cdots \leq G_n=G$$ with each $G_{i+1}/G_i\cong \Z/l\Z$.

\begin{proposition}\cite[p. 358]{MM}\label{prop} Let $(l,p-1)\neq 1)$ and $G$ a finite $l$-group with $d\left(G^{ab}\right)=s$.  Then for each $\mathbb{T}$ as in Lemma IV.10.10 \cite[p. 358]{MM}, of which there are infinitely many, there exists a geometric Galois extension $N/\F_p(t)$ with $$\Gal(N/\F_p(t))\cong G$$ and $$\ram(N/\F_p(t)) = \{\tilde{\mathfrak{q}_1},\ldots,\tilde{\mathfrak{q}_s}\},$$ where each $\tilde{\mathfrak{q}_i}\in \mathbb{P}(\F_p(t))$ are the uniquely determined extensions of $\mathfrak{q}_i\in \mathbb{T}$.
\end{proposition}

\begin{corollary}\label{LG1} Let $(l,p-1)\neq 1$.  Then every finite $l$-group satisfies Conjecture \ref{Restricted IGP}.
\end{corollary}
\begin{proof}
Existence is based on Proposition \ref{prop}, and minimality follows from Corollary \ref{minimality}.
\end{proof}

\begin{lemma}\cite[p. 353]{MM}\label{l1} Let $K=\F_p(t)$, $l\neq p$ a prime with $(l,p-1)=1$, and $\mathbb{S}\subset \mathbb{P}(K)$ a finite subset.  Then we have:
\begin{enumerate}
\item Every split central (geometric) $l^m$-Scholz embedding problem $\varepsilon (\phi, \kappa)$ over $K$ with kernel $\Z/l\Z$ and $exp_l(\phi(G_K))<l^m$ possesses a proper (geometric) $l^m$-Scholz solution.
\item If $\phi(G_K)$ is an $l$-Group and the socle of $N/K$ of the fixed field $N$ of $\phi(G_K)$ is unramified outside of $\mathbb{S}$ only, then $\varepsilon(\phi,\kappa)$ possesses also such a solution.  Moreover, its solution field $\tilde{N}$ satisfies $$\ram(\tilde{N}/K)\subseteq \ram(N/K) \cup \{\mathfrak{q}\}$$ for some $\mathfrak{q}\in \mathbb{P}(K)\backslash \mathbb{S}$.
\end{enumerate}
\end{lemma}

Note, the choice of $\mathfrak{q}$ guarantees the new extension will also be Scholz.

\begin{lemma}\cite[p. 354]{MM} \label{l2} Let $K$ and $\mathbb{S}$ be as in \ref{l1}.  Then every (geometric) non-split central $l^m$-Scholz embedding problem $\varepsilon(\phi,\kappa)$ over $K$ with kernel $\Z/l\Z$ and $exp_l(\phi(G_K))<l^m$ possesses a proper (geometric) solution, where the solution field $\tilde{N}$ satisfies $$\ram(\tilde{N}/K)= \ram(N/K).$$
\end{lemma}

\begin{lemma}\label{l3} Let $K$ and $\mathbb{S}$ be as in \ref{l1}.  Then every (geometric) non-split central $l^m$-Scholz embedding problem $\varepsilon(\phi,\kappa)$ over $K$ with kernel $\Z/l\Z$ and $|\phi(G_K)|=l^n<l^m$ possesses a proper (geometric) $l^m$-Scholz solution, where the solution field $\tilde{N}$ satisfies $$\ram(\tilde{N}/K)= \ram(N/K).$$
\end{lemma}

\begin{proof}
By Lemma \ref{l2} we already have a proper solution $\tilde{\phi}$ whose solution field $\tilde{N}$ satisfies the requisite ramification condition.  We will modify this solution to obtain a Scholz solution.

For $\mathfrak{r}\in \mathbb{T}:=\ram(\tilde{N}/K)$ with associated prime element $r$, fix prime divisors $\mathfrak{R}\in \mathbb{P}(N)$ and $\tilde{\mathfrak{R}}\in \mathbb{P}(\tilde{N})$.  Since $I(\mathfrak{R}/\mathfrak{r})=D(\mathfrak{R}/\mathfrak{r})$ and $\ker \phi \cong \Z/l\Z$, the decomposition group $D(\tilde{\mathfrak{R}}/\mathfrak{r})$ is contained in the preimage $\tilde{I}$ of type $\Z/l\Z \cdot I(\mathfrak{R}/\mathfrak{r})$ in $\tilde{G}$.  Thus it remains to show that $\tilde{I}$ is cyclic.

By Proposition IV.10.3 \cite[p. 351]{MM} $$\tilde{G}\cong D(\tilde{\mathfrak{R}}/\mathfrak{r})$$ which is cyclic in the non-split case.  Note, this is in fact a Frattini embedding problem, thus the socle remains the same.
\end{proof}

\begin{corollary}\label{l} Let $l$ be a prime with $(l,p-1)=1$.  Then every $l$-group satisfies Conjecture \ref{Restricted IGP}.
\end{corollary}
\begin{proof}
Minimality follows from Corollary \ref{minimality}.  For existence, we start with a cyclic extension ramified at one prime where we choose $M\in \F_p[t]$ such that $$\Phi(M)\equiv 0 \pmod{(p-1)l}$$ and thus also $$\Phi(M)\equiv 0 \pmod{(p-1)l^m}$$ where $m$ is such that $|G|=l^n<l^m$ and $\deg(M)=d$ such that $l^m|(p^d-1)$.  This guarantees that the base extension is Scholz.  Then using the decomposition for $l$-groups given above, and Lemma \ref{l3} we produce the desired extension by induction.
\end{proof}

\subsection{Nilpotent groups}

\begin{theorem}\label{N}

Let $G$ be a nilpotent group.  Then $G$ satisfies Conjecture \ref{Restricted IGP}.

\end{theorem}

\begin{proof}
$G$ is the direct product of its (unique) Sylow subgroups, thus the result follows by taking the compositum of the corresponding $G_l$-extensions found in Corollary \ref{l} and Proposition \ref{prop} and the $G_p$-extension found in Theorem \ref{p}.  To ensure proper ramification, start with the abelianization of $G$ as found in Theorem \ref{Abel} and use the cyclic subgroups of this as the starting point for the construction of each $l$-group.

\end{proof}

\section{Conclusion}

We have presented Conjecture \ref{Restricted IGP} describing the expected minimal ramification for $G$-extensions over the field $\F_p(t)$.  As evidence for this conjecture, we have established the cases where $G$ is a $p$-group (Theorem \ref{p}), abelian (Theorem \ref{Abel}), semiabelian with $d(G)=d(G^{ab})$ (Theorem \ref{semi}), is dihedral of order $2a$ where $p\neq 2,a$ and $a$ is an odd prime (Theorem \ref{Dn}), or is the symmetric group $S_3$ with $p\equiv 0,1\pmod{3}$ (Theorem \ref{S32}).  We then showed the case where $G$ is an $l$-group, $l\neq p$ (Corollaries \ref{LG1} and \ref{l}), and used this to prove the case of Nilpotent groups (Theorem \ref{N}).  We used methods ranging from generic polynomials and explicit computation, to the theory of pro-$p$ groups, to the theory of Drinfeld modules, to embedding theory.  These examples also illustrate several different ways of piecing together new cases from those already known.

These examples, together with the previous work mentioned, lead us to believe that Conjecture \ref{Restricted IGP} will hold true in general.  Further, Boston and Markin \cite{BM} implies that we should be able to achieve a quantitative result describing how often these so-called minimal extensions appear among all $G$-extensions.

\vfill\eject

\bibliographystyle{plain}
\bibliography{main}

\end{document}